\documentclass[12pt]{article}%
\usepackage[intlimits]{amsmath}
\usepackage{amssymb}
\usepackage[T1]{fontenc}
\usepackage[sc]{mathpazo}
\usepackage{color}
\usepackage[colorlinks]{hyperref}
\usepackage{amsfonts}
\usepackage{graphicx}%
\definecolor {refcol}{RGB}{40,0,255}
\hypersetup{colorlinks=true,allcolors=refcol}
\makeatletter
\newfont{\footsc}{cmcsc10 at 8truept}
\newfont{\footbf}{cmbx10 at 8truept}
\newfont{\footrm}{cmr10 at 10truept}
\newtheorem{theorem}{Theorem}

\newtheorem{conjecture}[theorem]{Conjecture}

\newtheorem{lemma}[theorem]{Lemma}

\newtheorem{proposition}[theorem]{Proposition}

\newenvironment{proof}[1][Proof]{\noindent{\textbf {#1}  }}  {\hfill$\Box$\bigskip}
\begin{document}

\title{\textbf{Symmetric functions and the principal case of the Frankl-F\"{u}redi
conjecture}}
\author{V. Nikiforov\thanks{Department of Mathematical Sciences, University of
Memphis, Memphis TN 38152, USA. Email: \textit{vnikifrv@memphis.edu}}}
\date{}
\maketitle

\begin{abstract}
Let $r\geq3$ and $G$ be an $r$-uniform hypergraph with vertex set $\left\{
1,\ldots,n\right\}  $ and edge set $E$. Let%
\[
\mu\left(  G\right)  :=\max%
{\textstyle\sum\limits_{\left\{  i_{1},\ldots,i_{r}\right\}  \in E}}
x_{i_{1}}\cdots x_{i_{r}},
\]
where the maximum is taken over all nonnegative $x_{1},\ldots,x_{n}$ with
$x_{1}+\cdots+x_{n}=1.$

Let $t\geq r-1$ be the unique real number such that $\left\vert E\right\vert
=\binom{t}{r}$. It is shown that if $r\leq5$ or $t\geq4\left(  r-1\right)
\left(  r-2\right)  $, then%
\[
\mu\left(  G\right)  \leq t^{-r}\binom{t}{r}%
\]
with equality holding if and only if $t$ is an integer.

The proof is based on some new bounds on elementary symmetric
functions.$\medskip$



\end{abstract}

\section{Introduction and main results}

Let $G$ be an $r$-uniform hypergraph ($r$\emph{-graph} for short) with vertex
set $V$ and edge set $E$. Assume that $V=\left[  n\right]  :=\left\{
1,\ldots,n\right\}  $ and let $\mathbf{x}:=\left(  x_{1},\ldots,x_{n}\right)
\in\mathbb{R}^{n}$. Write $P_{G}\left(  \mathbf{x}\right)  $ for the
polynomial form of $G$
\[
P_{G}\left(  \mathbf{x}\right)  :=%
{\textstyle\sum\limits_{\left\{  i_{1},\ldots,i_{r}\right\}  \in E}}
x_{i_{1}}\cdots x_{i_{r}},
\]
and let%
\[
\mu\left(  G\right)  :=\max_{\Delta^{n-1}}P_{G}\left(  \mathbf{x}\right)  ,
\]
where $\Delta^{n-1}\subset\mathbb{R}^{n}$ is the standard simplex:
\[
\Delta^{n-1}=\left\{  \mu\left(  G\right)  :\left(  x_{1},\ldots,x_{n}\right)
:x_{1}\geq0,\ldots,x_{n}\geq0\text{ and }x_{1}+\cdots+x_{n}=1\right\}  .
\]

We call $\mu\left(  G\right)  $ the \emph{MS-index }of $G$ in honor of Motzkin
and Straus, who introduced and studied $\mu\left(  G\right)  $ for $2$-graphs
in \cite{MoSt65}\footnote{The MS-index is commonly called the
\emph{Lagrangian}, a misty term that denies credit to Motzkin and Straus,
while Lagrange has nothing to do with the concept. Besides, the term
\emph{Lagrangian} has seven or so other meanings already in use elsewhere.}.
Let us note that the MS-index has a long-standing history in extremal
hypergraph theory (see, e.g., \cite{FrFu89} and \cite{Kee11} for more detailed discussion.)

Now, let
\[
\mu_{r}\left(  m\right)  :=\max\left\{  \mu\left(  G\right)  :G\text{ is an
}r\text{-graph with }m\text{ edges}\right\}  .
\]

The problem of finding $\mu_{r}\left(  m\right)  $ was first raised in 1989,
by Frankl and F\"{u}redi \cite{FrFu89}, who conjectured the exact value of
$\mu_{r}\left(  m\right)  $. During the years, their conjecture proved to be
rather hard: notwithstanding that it has been confirmed for most values of $m$
(see \cite{Tal02},\cite{TPWP16},\cite{Tyo17}), its toughest and most delicate
cases are still open.

However, even if completely solved, Frankl and F\"{u}redi's conjecture does
not provide an easy-to-use, closed-form expression for $\mu_{r}\left(
m\right)  $. In this regard, the following conjecture might be of interest:

\begin{conjecture}
\label{con1}Let $r\geq3$ and $G$ be an $r$-graph with $m$ edges. If $t\geq
r-1$ is the unique real number such that $m=\binom{t}{r}$, then $\mu\left(
G\right)  \leq mt^{-r},$ with equality if and only if $t$ is an integer.
\end{conjecture}

Note that the value of $\mu_{r}\left(  m\right)  $ conjectured by Frankl and
F\"{u}redi is quite close to $mt^{-r}$, and moreover, both values coincide if
$t$ is an integer. Tyomkyn \cite{Tyo17} called the latter case the\emph{
principal case\ }of the Frankl-F\"{u}redi's conjecture, and solved it for any
$r\geq4$ and $m$ sufficiently large; prior to that, Talbot \cite{Tal02} had
solved the principal case for $r=3$ and any $m$. Let us note that Talbot and
Tyomkyn contribute much more than the mentioned results, but neither of these
works imply a complete solution to Conjecture \ref{con1} for any $r$.

In this paper, we confirm Conjecture \ref{con1} whenever $3\leq r\leq5,$
thereby completely resolving the principal case of the Frankl-F\"{u}redi's
conjecture for these values of $r$. In addition, we show that Conjecture
\ref{con1} holds whenever $t\geq4\left(  r-1\right)  \left(  r-2\right)  $,
thereby giving an alternative proof of Tyomkyn's result and providing an
explicit bound\footnote{This bound is chosen for simplicity, and can be cut at
least by half. In contrast, Tyomkyn's proof of the principal case of the
Frankl-F\"{u}redi's conjecture provides no explicit bounds.}.

Our proofs are based on some seemingly novel bounds on elementary symmetric
functions, which, somewhat surprisingly, are just analytic results with no
relation to hypergraphs whatsoever. Theorems \ref{tub}, \ref{tlb}, and
\ref{tlub} below present the gist of this approach.\medskip

Given a vector $\mathbf{x}:=\left(  x_{1},\ldots,x_{n}\right)  $, write
$S_{k}\left(  \mathbf{x}\right)  $ for the $k$th elementary symmetric function
of $x_{1},\ldots,x_{n}$. Set $q\left(  \mathbf{x}\right)  :=x_{1}^{2}%
+\cdots+x_{n}^{2},$ and for nonnegative $\mathbf{x}$ set $\sigma\left(
\mathbf{x}\right)  :=x_{1}^{x_{1}}\cdots x_{n}^{x_{n}}$, with the caveat that
$0^{0}=1$.

Let us recall the most important case of the celebrated Maclaurin inequality
(\cite{M1729}, see also \cite{HLP88}, Theorem 52), that can be stated as:
\medskip

\emph{If }$k\geq2$\emph{ and }$\mathbf{x}\in\Delta^{n-1}$\emph{, then}
\begin{equation}
S_{k}\left(  \mathbf{x}\right)  <n^{-k}\binom{n}{k}, \label{MI}%
\end{equation}
\emph{unless the entries of }$\mathbf{x}$\emph{ are equal.\medskip}

Since $\mathbf{x}\in\Delta^{n-1}$ implies that $\sigma\left(  \mathbf{x}%
\right)  \geq1/n$, the following theorem strengthens inequality (\ref{MI})
under a mild restriction on the maximum entry of $\mathbf{x}$ (denoted by
$\left\vert \mathbf{x}\right\vert _{\max}$ hereafter):

\begin{theorem}
\label{tub}Let $k\geq3,$ $\mathbf{x}\in\Delta^{n-1}$, and $\sigma
=\sigma\left(  \mathbf{x}\right)  .$ If $\left\vert \mathbf{x}\right\vert
_{\max}\leq1/4\left(  k-2\right)  $, then
\[
S_{k}\left(  \mathbf{x}\right)  <\sigma^{k}\binom{1/\sigma}{k},
\]
unless the nonzero entries of $\mathbf{x}$ are equal.
\end{theorem}

It turns out that the bound in Theorem \ref{tub} is not an isolated exception,
but one of many interrelated similar bounds that avoid the parameter $n$
altogether (see the closing remarks of Section \ref{sym}.) In particular,
Theorem \ref{tub} can be matched with a very similar lower bound:

\begin{theorem}
\label{tlb}Let $k\geq3,$ $\mathbf{x}\in\Delta^{n-1}$, and $q=q\left(
\mathbf{x}\right)  $. If $\left\vert \mathbf{x}\right\vert _{\max}<1/\left(
k-1\right)  $, then
\[
S_{k}\left(  \mathbf{x}\right)  >q^{k}\binom{1/q}{k},
\]
unless the nonzero entries of $\mathbf{x}$ are equal.
\end{theorem}

The following theorem, crucial for tackling Conjecture \ref{con1},
incorporates an additional twist in order to weaken the constraint on
$\left\vert \mathbf{x}\right\vert _{\max}$:

\begin{theorem}
\label{tlub}Let $3\leq k\leq5,$ $\mathbf{x}:=\left(  x_{1},\ldots
,x_{n}\right)  \in\Delta^{n-1}$, and $\sigma=\sigma\left(  \mathbf{x}\right)
.$ If $x_{n}\leq\cdots\leq x_{1}\leq1/k$, then%
\begin{equation}
\frac{\partial S_{k}\left(  \mathbf{x}\right)  }{\partial x_{1}}<k\sigma
^{k}\binom{1/\sigma}{k}, \label{inr3}%
\end{equation}
unless the nonzero entries of $\mathbf{x}$ are equal.
\end{theorem}

The rest of the paper is structured as follows: in Section \ref{sym}, we give
some results about symmetric functions and prove Theorems \ref{tub},
\ref{tlb}, and \ref{tlub}; at the end of the section we discuss possible
extensions of these theorems. In Section \ref{pfc}, we prove an upper bound on
$\mu\left(  G\right)  $ and then prove Conjecture \ref{con1} if $3\leq r\leq5$
or $t\geq4\left(  r-1\right)  \left(  r-2\right)  $.

\section{\label{sym}Some bounds on elementary symmetric functions}

Let $\mathbf{x}:=\left(  x_{1},\ldots,x_{n}\right)  \in\Delta^{n-1}$, and
assume hereafter that $x_{1}\geq\cdots\geq x_{n}.$ Set
\[
p\left(  \mathbf{x}\right)  :=x_{1}^{3}+\cdots+x_{n}^{3}\text{ \ \ and
\ \ }t\left(  \mathbf{x}\right)  :=x_{1}^{4}+\cdots+x_{n}^{4}.
\]
Whenever $\mathbf{x}$ is understood, we shorten $S_{k}\left(  \mathbf{x}%
\right)  ,$ $q\left(  \mathbf{x}\right)  ,$ $p\left(  \mathbf{x}\right)  ,$
$t\left(  \mathbf{x}\right)  ,$ $\sigma\left(  \mathbf{x}\right)  $ to
$S_{k},$ $q,$ $p,$ $t,$ $\sigma$.

We start with a few basic inequalities about $\sigma,$ $q,$ $p,$ $t,$ and
$\left\vert \mathbf{x}\right\vert _{\max}$.

\begin{proposition}
\label{pmi}Let $\mathbf{x}\in\Delta^{n-1}$, and let $n^{\prime}$ be the number
of nonzero entries of $\mathbf{x}$. Then%
\[
1/n^{\prime}<\sigma<q<p^{1/2}<t^{1/3}<\left\vert \mathbf{x}\right\vert _{\max
},
\]
unless the nonzero entries of $\mathbf{x}$ are equal, in which case,
equalities hold throughout.
\end{proposition}

\begin{proof}
Without loss of generality, assume that $\mathbf{x}$ is positive, and that the
entries of $\mathbf{x}$ are not all equal. First, the function $x^{c}$ is
strictly convex for $x>0$ and $c>1$; hence, Jensen's inequality implies that%
\[%
{\textstyle\sum\limits_{i=1}^{n}}
x_{i}x_{i}<\left(
{\textstyle\sum\limits_{i=1}^{n}}
x_{i}x_{i}^{2}\right)  ^{1/2}<\left(
{\textstyle\sum\limits_{i=1}^{n}}
x_{i}x_{i}^{3}\right)  ^{1/3},
\]
yielding $q<p^{1/2}<t^{1/3}$. Likewise, since the function $\log
x$\footnote{Here and elsewhere $\log$ stands for \textquotedblleft logarithm
base $e$\textquotedblright.} is strictly concave for $x>0$, we see that
\[
\log\sigma=%
{\textstyle\sum\limits_{i=1}^{n}}
x_{i}\log x_{i}<\log\left(
{\textstyle\sum\limits_{i=1}^{n}}
x_{i}x_{i}\right)  =\log q,
\]
yielding $\sigma<q$. Further, since the function $x\log x$ is strictly convex
for $x>0$, we see that%
\[
\log\sigma=x_{1}\log x_{1}+\cdots+x_{n}\log x_{n}>n\left(  \frac{1}{n}%
\log\frac{1}{n}\right)  =\log1/n,
\]
yielding $\sigma>1/n$. Finally,
\[
t=x_{1}^{4}+\cdots+x_{n}^{4}<\left\vert \mathbf{x}\right\vert _{\max}%
^{3}\left(  x_{1}+\cdots+x_{n}\right)  =\left\vert \mathbf{x}\right\vert
_{\max}^{3},
\]
completing the proof of Proposition \ref{pmi}.
\end{proof}

Further, note that $\partial S_{k}/\partial x_{i}$ is just the sum of all
products in $S_{k-1}$ that do not contain $x_{i}$; thus for every $i\in\left[
n\right]  $, we have%
\begin{equation}
\frac{\partial S_{k}}{\partial x_{i}}=S_{k-1}-x_{i}\frac{\partial S_{k-1}%
}{\partial x_{i}}. \label{MR}%
\end{equation}
In addition, a short argument shows that
\[
\frac{\partial S_{k}}{\partial x_{i}}-\frac{\partial S_{k}}{\partial x_{j}%
}=\left(  x_{j}-x_{i}\right)  \frac{\partial S_{k}}{\partial x_{i}\partial
x_{j}}.
\]
Hence, the assumption $x_{1}\geq\cdots\geq x_{n}$ implies that $\partial
S_{k}/\partial x_{1}\leq\cdots\leq\partial S_{k}/\partial x_{n}$. Moreover, we
see that $x_{i}=x_{j}$ if and only if $\partial S_{k}/\partial x_{i}=\partial
S_{k}/\partial x_{j}.$\medskip

Our proofs crucially rely on the weighted Chebyshev inequality (see, e.g.,
\cite{Bul03}, p. 161):\medskip

\textbf{Chebyshev's inequality. }\emph{Let }$\mathbf{w}:=\left(  w_{1}%
,\ldots,w_{n}\right)  \in\Delta^{n-1}$\emph{ be positive, and let }$a_{1}%
\leq\cdots\leq a_{n\text{ }}$\emph{. If }$b_{1}\geq\cdots\geq b_{n},$\emph{
then }%
\[%
{\textstyle\sum\limits_{i=1}^{n}}
w_{i}a_{i}b_{i}\leq%
{\textstyle\sum\limits_{i=1}^{n}}
w_{i}a_{i}%
{\textstyle\sum\limits_{i=1}^{n}}
w_{i}b_{i}.
\]
\emph{If }$b_{1}\leq\cdots\leq b_{n}$\emph{, then the opposite inequality
holds. In both cases equality holds if and only if }$a_{1}=\cdots=a_{n}$\emph{
or }$b_{1}=\cdots=b_{n}$\emph{.}

\subsection{Two recurrence inequalities}

The proof of Theorems \ref{tub} and \ref{tlb} reside on two recurrence
inequalities, stated in Propositions \ref{glb} and \ref{gub} below.

\begin{proposition}
\label{glb}If $k\geq2$ and $\mathbf{x}\in\Delta^{n-1}$, then%
\begin{equation}
kS_{k}\geq\left(  1-\left(  k-1\right)  q\right)  S_{k-1}. \label{ina}%
\end{equation}
Equality holds if and only if the nonzero entries of $\mathbf{x}$ are equal.
\end{proposition}

\begin{proof}
Without loss of generality, we assume that $\mathbf{x}$ is positive, for
dropping out its zero entries does not alter $S_{2},\ldots,S_{n},$ and $q$.

First, multiplying equation (\ref{MR}) by $x_{i}$ and summing the results, we
get%
\[
kS_{k}=%
{\textstyle\sum\limits_{i=1}^{n}}
x_{i}\frac{\partial S_{k}}{\partial x_{i}}=%
{\textstyle\sum\limits_{i=1}^{n}}
x_{i}S_{k-1}-%
{\textstyle\sum\limits_{i=1}^{n}}
x_{i}^{2}\frac{\partial S_{k-1}}{\partial x_{i}}=S_{k-1}-%
{\textstyle\sum\limits_{i=1}^{n}}
x_{i}^{2}\frac{\partial S_{k-1}}{\partial x_{i}}.
\]
Now, let $w_{i}=a_{i}=x_{i}$ and $b_{i}=\partial S_{k-1}/\partial x_{i}$ for
all $i\in\left[  n\right]  ,$ and note that Chebyshev's inequality implies
that
\begin{equation}%
{\textstyle\sum\limits_{i=1}^{n}}
x_{i}x_{i}\frac{\partial S_{k-1}}{\partial x_{i}}\leq%
{\textstyle\sum\limits_{i=1}^{n}}
x_{i}^{2}%
{\textstyle\sum\limits_{i=1}^{n}}
x_{i}\frac{\partial S_{k-1}}{\partial x_{i}}=\left(  k-1\right)  qS_{k-1},
\label{ina1}%
\end{equation}
so inequality (\ref{ina}) follows.

The sufficiency of the condition for equality in (\ref{ina}) is clear, so we
only prove its necessity. If equality holds in (\ref{ina}), then equality
holds in (\ref{ina1}). Hence, the condition for equality in Chebyshev's
inequality implies that
\[
x_{1}=\cdots=x_{n}\text{ \ \ or \ \ }\frac{\partial S_{k-1}}{\partial x_{1}%
}=\cdots=\frac{\partial S_{k-1}}{\partial x_{n}}.
\]
As noted above, in either case $x_{1}=\cdots=x_{n}$, completing the proof of
Proposition \ref{glb}.
\end{proof}

\medskip

\begin{proposition}
\label{gub}If $k\geq3$ and $\mathbf{x}\in\Delta^{n-1}$, then%
\begin{equation}
kS_{k}\leq S_{k-1}-\left(  q-\left(  k-2\right)  p\right)  S_{k-2}.
\label{inb}%
\end{equation}
Equality holds if and only if the nonzero entries of $\mathbf{x}$ are equal.
\end{proposition}

\begin{proof}
Without loss of generality we assume that $\mathbf{x}$ is positive. As in the
proof of Proposition \ref{glb}, we see that
\[
kS_{k}=S_{k-1}-%
{\textstyle\sum\limits_{i=1}^{n}}
x_{i}^{2}\frac{\partial S_{k-1}}{\partial x_{i}}=S_{k-1}-%
{\textstyle\sum\limits_{i=1}^{n}}
x_{i}^{2}\left(  S_{k-2}-x_{i}\frac{\partial S_{k-2}}{\partial x_{i}}\right)
=S_{k-1}-qS_{k-2}+%
{\textstyle\sum\limits_{i=1}^{n}}
x_{i}^{3}\frac{\partial S_{k-2}}{\partial x_{i}}.
\]
Now, let $w_{i}=x_{i},$ $a_{i}=x_{i}^{2}$ and $b_{i}=\partial S_{k-2}/\partial
x_{i}$ for all $i\in\left[  n\right]  ,$ and note that Chebyshev's inequality
implies that
\begin{equation}%
{\textstyle\sum\limits_{i=1}^{n}}
x_{i}^{3}\frac{\partial S_{k-2}}{\partial x_{i}}=%
{\textstyle\sum\limits_{i=1}^{n}}
x_{i}x_{i}^{2}\frac{\partial S_{k-2}}{\partial x_{i}}\leq%
{\textstyle\sum\limits_{i=1}^{n}}
x_{i}^{3}%
{\textstyle\sum\limits_{i=1}^{n}}
x_{i}\frac{\partial S_{k-2}}{\partial x_{i}}=\left(  k-2\right)  pS_{k-2}.
\label{inb1}%
\end{equation}
so inequality (\ref{inb}) follows.

The sufficiency of the condition for equality in (\ref{inb}) is clear, so we
only prove its necessity. If equality holds in (\ref{inb}), then equality
holds in (\ref{inb1}). Hence, the condition for equality in Chebyshev's
inequality implies that
\[
x_{1}^{2}=\cdots=x_{n}^{2}\text{ \ \ or \ \ }\frac{\partial S_{k-2}}{\partial
x_{1}}=\cdots=\frac{\partial S_{k-2}}{\partial x_{n}}.
\]
In either case $x_{1}=\cdots=x_{n}$, completing the proof of Proposition
\ref{gub}.
\end{proof}

\subsection{Proofs of Theorems \ref{tub} and \ref{tlb}}

\begin{proof}
[\textbf{Proof of Theorem \ref{tub}.}]Set for short $x=x_{1}.$ Our proof
hinges on two claims:\medskip

\textbf{Claim 1. }\emph{If }$x\leq1/4\left(  k-2\right)  ,$ then for
$i=3,\ldots,k,$ we have%
\[
iS_{i}\leq S_{i-1}-\sigma\left(  1-\left(  i-2\right)  \sigma\right)
S_{i-2}.
\]

\emph{Proof.} Referring to Proposition \ref{gub}, it is enough to show that
\[
q-\left(  i-2\right)  p\geq\sigma\left(  1-\left(  i-2\right)  \sigma\right)
.
\]
To this end, let%
\[
f\left(  z\right)  :=e^{z}-\left(  i-2\right)  e^{2z},
\]
and note that $f\left(  z\right)  $ is convex whenever $e^{z}\leq1/4\left(
i-2\right)  .$ Hence, in view of $x\leq1/4\left(  k-2\right)  $,%
\[
q-\left(  i-2\right)  p=%
{\textstyle\sum\limits_{j=1}^{n}}
x_{j}f\left(  \log x_{j}\right)  \leq f\left(
{\textstyle\sum\limits_{j=1}^{n}}
x_{j}\log x_{j}\right)  =\sigma-\left(  i-2\right)  \sigma^{2},
\]
proving Claim 1.\medskip

\textbf{Claim 2. }\emph{If }$x\leq1/4\left(  k-2\right)  ,$ then for
$i=3,\ldots,k,$ we have
\begin{equation}
iS_{i}\leq\left(  1-\left(  i-1\right)  \sigma\right)  S_{i-1}. \label{reci}%
\end{equation}

\emph{Proof. }We use induction on $i$. If $i=3,$ then Claim 1 yields
\[
3S_{3}=S_{2}-\left(  \sigma-\sigma^{2}\right)  =S_{2}-\sigma\left(
1-q\right)  =\left(  1-2\sigma\right)  S_{2}\text{;}%
\]
hence, the statement holds for $i=3.$ If $i>3$, then the induction assumption
implies that
\[
\frac{\left(  i-1\right)  S_{i-1}}{1-\left(  i-2\right)  \sigma}\geq\left(
i-1\right)  S_{i-1},
\]
because $1-\left(  i-2\right)  \sigma\geq1-\left(  k-2\right)  x>0.$ Now,
Claim 1 yields
\begin{align*}
iS_{k}  &  \leq S_{i-1}-\sigma\left(  1-\left(  i-2\right)  \sigma\right)
S_{i-2}\leq S_{i-1}-\sigma\left(  1-\left(  i-2\right)  \sigma\right)
\frac{\left(  i-1\right)  S_{i-1}}{\left(  1-\left(  i-2\right)
\sigma\right)  }\\
&  =\left(  1-\left(  i-1\right)  \sigma\right)  S_{i-1},
\end{align*}
completing the induction step and the proof of Claim 2.\medskip

To finish the proof of Theorem \ref{tub}, note that $1-\left(  i-1\right)
\sigma\geq1-\left(  k-1\right)  x>0.$ Thus we can multiply inequalities
(\ref{reci}) for $i=3,\ldots,k,$ getting%
\[
\frac{k!}{2}S_{3}\cdots S_{k}\leq\left(  1-2\sigma\right)  \cdots\left(
1-\left(  k-1\right)  \sigma\right)  S_{2}\cdots S_{k-1}.
\]
Now, using the fact $2S_{2}=1-q\leq1-\sigma$, we see that
\[
S_{k}\leq\frac{1}{k!}\left(  1-\sigma\right)  \left(  1-2\sigma\right)
\cdots\left(  1-\left(  k-1\right)  \sigma\right)  =\sigma^{k}\binom{1/\sigma
}{k},
\]
as desired. The condition for equality in this inequality follows from the
conditions for equality in Proposition \ref{gub}.Theorem \ref{tub} is proved.
\end{proof}

\emph{Remark. }The proof of Theorem \ref{tub} shows that its conclusion can be
strengthened to%
\[
S_{k}\leq\frac{1}{k!}\left(  1-q\right)  \left(  1-2\sigma\right)
\cdots\left(  1-\left(  k-1\right)  \sigma\right)  .
\]
\medskip

\begin{proof}
[\textbf{Proof of Theorem \ref{tlb}.}]Proposition \ref{glb} implies that
\begin{equation}
S_{i}\geq\frac{1}{i}\left(  1-\left(  i-1\right)  q\right)  S_{i-1}
\label{lin}%
\end{equation}
for every $i=2,\ldots,k$. Since
\[
1-\left(  i-1\right)  q\geq1-\left(  k-1\right)  \left\vert \mathbf{x}%
\right\vert _{\max}>1-\frac{k-1}{k-1}=0,
\]
we can multiply inequalities (\ref{lin}) for $i=2,\ldots,k,$ obtaining%
\[
S_{2}\cdots S_{k}\geq\frac{1}{k!}S_{1}\cdots S_{k-1}\left(  1-q\right)
\cdots\left(  1-\left(  k-1\right)  q\right)  =q^{k}\binom{1/q}{k},
\]
as desired.

The condition for equality in this inequality follows from the conditions for
equality in Proposition \ref{glb}. Theorem \ref{tlb} is proved.
\end{proof}

\subsection{Proof of Theorem \ref{tlub}}

The proof of Theorem \ref{tlub} is the most involved one in this paper,
especially the case $k=5$. We give separate proofs for $k=3,4,5$, because a
compound one would be a harder read.

In all three cases we assume that $\mathbf{x}$ is positive, and set $x=x_{1}$.
The proofs of the conditions for equality are straightforward and are
omitted.\medskip

\begin{proof}
[\textbf{Proof for }$k=3$.]Using equation (\ref{MR}) and Proposition
\ref{pmi}, we find that
\[
\frac{\partial S_{3}}{\partial x_{1}}=S_{2}-x\frac{\partial S_{2}}{\partial
x_{1}}=\frac{1-q}{2}-x\left(  1-x\right)  \leq\frac{1-\sigma-2x\left(
1-x\right)  }{2}.
\]
Since $x\left(  1-x\right)  $ is increasing for $x\leq1/2,$ it follows that
$\sigma\left(  1-\sigma\right)  \leq x\left(  1-x\right)  ,$ and so
\[
\frac{\partial S_{3}}{\partial x_{1}}\leq\frac{1-\sigma-2x\left(  1-x\right)
}{2}\leq\frac{1-\sigma-2\sigma\left(  1-\sigma\right)  }{2}=3\sigma^{3}%
\binom{1/\sigma}{3}.
\]
Theorem \ref{tlub} is proved for $k=3.$
\end{proof}

\medskip

\begin{proof}
[\textbf{Proof for }$k=4$.]Equation (\ref{MR}) implies that%
\begin{align*}
\frac{\partial S_{4}}{\partial x_{1}}  &  =S_{3}-x\frac{\partial S_{3}%
}{\partial x_{1}}=S_{3}-x\left(  S_{2}-x\frac{\partial S_{2}}{\partial x_{1}%
}\right)  =\frac{1-3q+2p}{6}-x\left(  \frac{1-q}{2}-x+x^{2}\right) \\
&  =\frac{1}{6}\left(  1-3q+2p-3x\left(  1-q\right)  +6x^{2}-6x^{3}\right)
\end{align*}

To establish (\ref{inr3}), we prove a chain of inequalities, consecutively
eliminating the parameters $p,$ $q,$ and $x$ from the right side of the above equation.

First, note that the function
\[
f\left(  z\right)  =e^{z}-e^{2z}%
\]
is convex whenever $e^{z}\leq1/4.$ Hence,%
\[
q-p=%
{\textstyle\sum\limits_{i=1}^{n}}
x_{i}f\left(  \log x_{i}\right)  \geq f\left(
{\textstyle\sum\limits_{i=1}^{n}}
x_{i}\log x_{i}\right)  =\sigma-\sigma^{2},
\]
yielding in turn%
\begin{align*}
6\frac{\partial S_{4}}{\partial x_{1}}  &  \leq1-q-2\sigma+2\sigma
^{2}-3x\left(  1-q\right)  +6x^{2}-6x^{3}\\
&  =1-q\left(  1-3x\right)  -2\sigma+2\sigma^{2}-3x+6x^{2}-6x^{3}.
\end{align*}
Since $1-3x>0$ and $q\leq\sigma,$ we get%
\[
6\frac{\partial S_{4}}{\partial x_{1}}\leq1-3x\sigma-3\sigma+2\sigma
^{2}-3x+6x^{2}-6x^{3}.
\]
To finish the proof, we have to show that
\[
1-3x\sigma-3\sigma+2\sigma^{2}-3x+6x^{2}-6x^{3}\leq1-6\sigma+11\sigma
^{2}-6\sigma^{3}=24\sigma^{4}\binom{1/\sigma}{4},
\]
which, after some algebra, turns out to be equivalent to
\[
\left(  x-\sigma\right)  \left(  3x-2x^{2}+2\sigma-2x\sigma-2\sigma
^{2}\right)  \leq\left(  x-\sigma\right)  .
\]
Using the fact $x\leq\sigma,$ if suffice to prove that
\[
3x-2x^{2}+2\sigma-2x\sigma-2\sigma^{2}<1.
\]
However, $2\sigma-2x\sigma-2\sigma^{2}$ increases in $\sigma$ because
$\sigma\leq x\leq1/4<\left(  1-x\right)  /2.$ Hence,%
\[
3x-2x^{2}+2\sigma-2x\sigma-2\sigma^{2}\leq5x-6x^{2}\leq\frac{5}{4}-\frac
{6}{16}<1.
\]
Theorem \ref{tlub} is proved for $k=4.$
\end{proof}

\medskip

\begin{proof}
[\textbf{Proof for} $k=5$.]Equation (\ref{MR}) implies that%
\begin{align}
\frac{\partial S_{5}}{\partial x_{1}}  &  =S_{4}-x\frac{\partial S_{4}%
}{\partial x_{1}}=S_{4}-xS_{3}+x^{2}\frac{\partial S_{3}}{\partial x_{1}%
}=S_{4}-xS_{3}+x^{2}S_{2}-x^{3}\frac{\partial S_{2}}{\partial x_{1}%
}\nonumber\\
&  =\frac{1-6q+3q^{2}+8p-6t}{24}-x\frac{1-3q+2p}{6}+x^{2}\frac{1-q}{2}%
-x^{3}\left(  1-x\right) \nonumber\\
&  =\frac{1}{24}\left(  1-6q+3q^{2}+8\left(  1-x\right)  p-6t+12xq-12x^{2}%
q-4x+12x^{2}-24x^{3}+24x^{4}\right)  . \label{me5}%
\end{align}

To establish (\ref{inr3}), we prove a chain of inequalities, consecutively
eliminating the parameters $t,p,q,$ and $x$ from the right side of equation
(\ref{me5})\footnote{The reader may find that the tradeoffs between the stages
of the calculations result in weird numbers, but we were not able to spare
much room for elegance.}.

For a start, the following claim is used to eliminate $p$ and $t$:\medskip

\textbf{Claim 1 }\emph{If }$x<1/5,$ \emph{then }%
\begin{equation}
-\frac{106-160x}{25}q+8\left(  1-x\right)  p-6t\leq-\frac{106-160x}{25}%
\sigma+8\left(  1-x\right)  \sigma^{2}-6\sigma^{3}. \label{insig}%
\end{equation}

\emph{Proof.} Claim 1 follows from the fact that the function
\[
f\left(  z\right)  :=-\frac{106-160x}{25}e^{z}+8\left(  1-x\right)
e^{2z}-6e^{3z}%
\]
is concave whenever $e^{z}\leq1/5$. To verify this fact, we show that
\begin{equation}
f^{\prime\prime}\left(  z\right)  =e^{z}\left(  -\frac{106-160x}{25}+32\left(
1-x\right)  e^{z}-54e^{2z}\right)  \leq0. \label{concf}%
\end{equation}
Indeed, the expression $32\left(  1-x\right)  e^{z}-54e^{2z}$ is quadratic in
$e^{z}$, and thus it increases in $z$ whenever $e^{z}\leq8\left(  1-x\right)
/27$. On the other hand,
\[
e^{z}\leq1/5\leq8\left(  1-x\right)  /27,
\]
implying that
\[
32\left(  1-x\right)  e^{z}-54e^{2z}\leq\frac{32\left(  1-x\right)  }{5}%
-\frac{54}{25}=\frac{106-160x}{25}.
\]
The latter inequality clearly entails (\ref{concf}); therefore, $f\left(
z\right)  $ is concave.

Now, the concavity of $f\left(  z\right)  $ implies that%
\begin{align*}
-\frac{106-160x}{25}q+8\left(  1-x\right)  p-6t  &  =\sum x_{i}f\left(  \log
x_{i}\right)  \leq f\left(  x_{i}\log x_{i}\right) \\
&  =-\frac{106-160x}{25}\sigma+8\left(  1-x\right)  \sigma^{2}-6\sigma^{3},
\end{align*}
completing the proof of Claim 1.\medskip

To use Claim 1 we add and subtract the term $\frac{106-160x}{25}q$ in the
right side of (\ref{me5}), and note that
\[
-6q+\frac{106-160x}{25}q+12xq-12x^{2}q=-\frac{44-130x+300x^{2}}{25}q+\frac
{2}{5}xq.
\]
Thus, summarizing the current progress, Claim 1 implies that
\begin{align}
24\frac{\partial S_{5}}{\partial x_{1}}  &  \leq1-\frac{44-130x+300x^{2}}%
{25}q+3q^{2}+\frac{2}{5}xq-4x+12x^{2}-24x^{3}+24x^{4}\label{in1}\\
&  -\frac{106-160x}{25}\sigma+8\left(  1-x\right)  \sigma^{2}-6\sigma
^{3}.\nonumber
\end{align}

Our next goal is to eliminate $q$ in the right side of (\ref{in1}). To this
end, define the function
\[
g\left(  z\right)  :=-\frac{44-130x+300x^{2}}{25}z+3z^{2}%
\]
\medskip

\textbf{Claim 2} \emph{If }$x\leq1/5,$ \emph{then }$g\left(  q\right)  \leq
g\left(  \sigma\right)  .\medskip$

\emph{Proof. }Claim 2 follows from the fact that $g\left(  z\right)  $
decreases in $z$ whenever $z\leq x$. To prove this fact note that $g\left(
z\right)  $ is quadratic in $z$, and so it decreases whenever
\[
z\leq\frac{44-130x+300x^{2}}{150}.
\]
However, the stipulation $x\leq1/5$ entails that
\[
x\leq\frac{44-130x+300x^{2}}{150}.
\]
Thus, $g\left(  z\right)  $ decreases in $z$ whenever $z\leq x$. In
particular, the inequalities $\sigma\leq q\leq x$ imply that $g\left(
q\right)  \leq g\left(  \sigma\right)  $, completing the proof of Claim
2.\medskip

Applying Claim 2, we replace $q$ by $\sigma$ in the right side of (\ref{in1}),
and obtain
\begin{align*}
24\frac{\partial S_{5}}{\partial x_{1}}  &  \leq1-\frac{44-130x+300x^{2}}%
{25}\sigma+3\sigma^{2}-\frac{106-160x}{25}\sigma+8\left(  1-x\right)
\sigma^{2}-6\sigma^{3}+\frac{2}{5}xq\\
&  -4x+12x^{2}-24x^{3}+24x^{4}\\
&  =1-\frac{150-290x+300x^{2}}{25}\sigma+11\sigma^{2}-8x\sigma^{2}-6\sigma
^{3}-4x+\frac{62}{5}x^{2}-24x^{3}+24x^{4}\\
&  =1-6\sigma+11\sigma^{2}-6\sigma^{3}+\frac{58}{5}x\sigma-12x^{2}%
\sigma-8x\sigma^{2}-4x+\frac{62}{5}x^{2}-24x^{3}+24x^{4}.
\end{align*}
In the above derivation we also use the inequality $\frac{2}{5}xq\leq\frac
{2}{5}x^{2}$.

Therefore, to finish the proof of (\ref{inr3}), it remains to show that
\begin{align*}
&  1-6\sigma+11\sigma^{2}-6\sigma^{3}+\frac{58}{5}x\sigma-12x^{2}%
\sigma-8x\sigma^{2}-4x+\frac{62}{5}x^{2}-24x^{3}+24x^{4}\\
&  \leq1-10\sigma+35\sigma^{2}-50\sigma^{3}+24\sigma^{4}=120\sigma^{5}%
\binom{1/\sigma}{5},
\end{align*}
which is equivalent to%
\[
\frac{58}{5}x\sigma-12x^{2}\sigma-8x\sigma^{2}-4x+\frac{62}{5}x^{2}%
-24x^{3}+24x^{4}\leq-4\sigma+24\sigma^{2}-44\sigma^{3}+24\sigma^{4}.
\]
After rearranging and factoring $\left(  x-\sigma\right)  $ out, we get
\begin{align*}
4\left(  x-\sigma\right)   &  \geq\frac{58}{5}\left(  x-\sigma\right)
\sigma+\frac{62}{5}\left(  x-\sigma\right)  \left(  x+\sigma\right) \\
&  -12\left(  x-\sigma\right)  \left(  x+\sigma\right)  \sigma-8\left(
x-\sigma\right)  \sigma^{2}-24\left(  x-\sigma\right)  \left(  x^{2}%
+x\sigma+\sigma^{2}\right) \\
&  +24\left(  x-\sigma\right)  \left(  x^{3}+x^{2}\sigma+x\sigma^{2}%
+\sigma^{3}\right)  .
\end{align*}
Since $x\geq\sigma$, it suffices to show that%
\begin{align}
2  &  >\frac{29}{5}\sigma+\frac{31}{5}\left(  x+\sigma\right)  -6\left(
x+\sigma\right)  \sigma-4\sigma^{2}-12\left(  x^{2}+x\sigma+\sigma^{2}\right)
+12\left(  x^{3}+x^{2}\sigma+x\sigma^{2}+\sigma^{3}\right) \nonumber\\
&  =\frac{31}{5}x-12x^{2}+12x^{3}+12\sigma^{3}-\left(  22-12x\right)
\sigma^{2}+\left(  12-18x+12x^{2}\right)  \sigma. \label{exp2}%
\end{align}
To this end, set
\[
h\left(  z\right)  :=12z^{3}-\left(  22-12x\right)  z^{2}+\left(
12-18x+12x^{2}\right)  z,
\]
and note that
\[
h^{\prime}\left(  z\right)  =36z^{2}-2\left(  22-12x\right)  z+\left(
12-18x+12x^{2}\right)  .
\]
Since $h^{\prime}\left(  z\right)  $ is quadratic in $z$ and $36>0,$ we see
that $h\left(  z\right)  $ is increasing if $z\leq z_{\min}$, where $z_{\min}$
is the smaller root of the equation $h^{\prime}\left(  z\right)  =0.$ However,
the stipulation $x\leq1/5$ easily implies that
\[
z_{\min}=\frac{22-12x-\sqrt{\left(  22-12x\right)  ^{2}-36\left(
12-18x+12x^{2}\right)  }}{36}>x,
\]
and therefore $h\left(  z\right)  $ is increasing in $z$ if $z\leq x$.

Thus, in view of $\sigma\leq x\leq1/5,$ we find that
\begin{align*}
12\sigma-18x\sigma-22\sigma^{2}+12x^{2}\sigma+12x\sigma^{2}+12\sigma^{3}  &
=h\left(  \sigma\right) \\
&  \leq h\left(  x\right)  =12x-40x^{2}+36x^{3}.
\end{align*}
Finally, for the right side of (\ref{exp2}) we obtain
\[
\frac{31}{5}x-12x^{2}+12x^{3}+12\sigma^{3}-\left(  22-12x\right)  \sigma
^{2}+\left(  12-18x+12x^{2}\right)  \sigma\leq\frac{91}{5}x-52x^{2}%
+48x^{3}<2.
\]
Theorem \ref{tlub} is proved for $k=5$.
\end{proof}

\subsection{Closing remarks}

The restriction on $\left\vert \mathbf{x}\right\vert _{\max}$ in Theorem
\ref{tub} can be somewhat relaxed. For example, for $r=3$ it is enough to
require that $\left\vert \mathbf{x}\right\vert _{\max}\leq3/8,$ while for
$r=4$ it is enough to have $\left\vert \mathbf{x}\right\vert _{\max}\leq
11/48$. It is challenging to find the weakest possible restriction on
$\left\vert \mathbf{x}\right\vert _{\max}$ for the conclusion of Theorem
\ref{tub} to hold.

It is unlikely that Theorem \ref{tlub} remains valid as is for sufficiently
large $r$; even the case $r=6$ is a challenge. Thus, it is interesting what
alterations are necessary to prove Conjecture \ref{con1} for $r>5$. Here is a
possibility for some progress:

Given $\mathbf{x}\in\Delta^{n-1}$ and real $t\geq-1$, define%
\[
\varphi_{t}\left(  \mathbf{x}\right)  :=\left\{
\begin{array}
[c]{cc}%
\left(
{\textstyle\sum\limits_{i=1}^{n}}
x_{i}^{1+t}\right)  ^{1/t}\text{,} & \text{if }x\neq0\text{;}\\
\sigma\left(  \mathbf{x}\right)  , & \text{if }x=0.
\end{array}
\right.
\]
Note that if $\mathbf{x}$ is fixed and $t\in\left[  -1,\infty\right)  $, the
function $\varphi_{t}\left(  \mathbf{x}\right)  $ is continuous and
nondecreasing in $t.$ In addition, assuming that $0^{0}=1,$ we see that
\[
\varphi_{1}\left(  \mathbf{x}\right)  =q\left(  \mathbf{x}\right)  \text{;
\ \ \ }\varphi_{-1}\left(  \mathbf{x}\right)  =1/n\text{; \ \ \ }\lim
{}_{t\rightarrow\infty}\varphi_{t}\left(  \mathbf{x}\right)  =\left\vert
x\right\vert _{\max}.
\]

It seems possible to extend Theorems \ref{tub}, \ref{tlb}, and \ref{tlub}
using $\varphi_{t}\left(  \mathbf{x}\right)  $ instead of $q$ and $\sigma$.

\section{\label{pfc}Proof of Conjecture \ref{con1} for $3\leq r\leq5$ or
$t\geq4\left(  r-1\right)  \left(  r-2\right)  $}

Let $G$ be an $r$-graph of order $n$. A vector $\mathbf{x}\in\Delta^{n-1}$
such that $P_{G}\left(  \mathbf{x}\right)  =\mu\left(  G\right)  $ is called
an \emph{eigenvector} to $\mu\left(  G\right)  .$

Let $\mathbf{x}:=\left(  x_{1},\ldots,x_{n}\right)  \in\Delta^{n-1}$ be an
eigenvector to $\mu\left(  G\right)  $. Using Lagrange
multipliers\footnote{This argument is known from the times of Motzkin and
Straus, so we skip the details.}, one finds that
\begin{equation}
r\mu\left(  G\right)  =\frac{\partial P_{G}\left(  \mathbf{x}\right)
}{\partial x_{i}}=%
{\textstyle\sum\limits_{\left\{  i_{1},\ldots,i_{r-1},i\right\}  \in E}}
x_{i_{1}}\cdots x_{i_{r-1}}, \label{eeq}%
\end{equation}
for every $i\in\left[  n\right]  $ such that $x_{i}>0.$

We start with a simple lemma, valid for any $r\geq2$:

\begin{lemma}
\label{tsin}Let $G$ be an $r$-graph of order $n$, with $m$ edges. If
$\mathbf{x}\in\Delta^{n-1}$ is an eigenvector to $\mu\left(  G\right)  $, then%
\begin{equation}
\mu\left(  G\right)  \leq m\sigma^{r}. \label{sin}%
\end{equation}

\end{lemma}

\begin{proof}
Clearly the lemma holds if $m=0$, so suppose that $m>0$. Likewise, without
loss of generality, suppose that $\mathbf{x}$ is positive. Let $\mu=\mu\left(
G\right)  $ and note that equations (\ref{eeq}) imply that
\begin{align*}
\mu\log\sigma^{r}  &  =r\mu%
{\textstyle\sum\limits_{i=1}^{n}}
x_{i}\log x_{i}=%
{\textstyle\sum\limits_{i=1}^{n}}
\frac{\partial P\left(  \mathbf{x}\right)  }{\partial x_{i}}x_{i}\log x_{i}=%
{\textstyle\sum\limits_{\left\{  i_{1},\ldots,i_{r}\right\}  \in E}}
x_{i_{1}}\cdots x_{i_{r}}\left(  \log x_{i_{1}}+\cdots+\log x_{i_{r}}\right)
\\
&  =%
{\textstyle\sum\limits_{\left\{  i_{1},\ldots,i_{r}\right\}  \in E}}
x_{i_{1}}\cdots x_{i_{r}}\log x_{i_{1}}\cdots x_{i_{r}}.
\end{align*}
Since the function $x\log x$ is convex, we see that
\[%
{\textstyle\sum\limits_{\left\{  i_{1},\ldots,i_{r}\right\}  \in E}}
x_{i_{1}}\cdots x_{i_{r}}\log x_{i_{1}}\cdots x_{i_{r}}\geq m\left(  \frac
{\mu}{m}\right)  \log\frac{\mu}{m}=\mu\log\frac{\mu}{m}.
\]
Hence,%
\[
\mu\log\sigma^{r}\geq\mu\log\frac{\mu}{m},
\]
implying that $\mu\leq m\sigma^{r}$, as desired.
\end{proof}

\emph{Remark. }Note that $\mu\left(  G\right)  \leq mq^{r}$ is a weaker, yet
more usable consequence of bound (\ref{sin}).\medskip

With Theorem \ref{tlub} and Lemma \ref{tsin} in hand, it is not hard to prove
Conjecture \ref{con1} for $3\leq r\leq5$:$\medskip$

\begin{proof}
[\textbf{Proof of Conjecture \ref{con1} for }$3\leq r\leq5.$]Let $G$ be an
$r$-graph of order $n$ with $m$ edges, and let $\mu\left(  G\right)  =\mu
_{r}\left(  m\right)  .$ Suppose that $t>r-1$ is a real number satisfying
$m=\binom{t}{r}$. To prove Conjecture \ref{con1}, we have to show that
\begin{equation}
\mu\left(  G\right)  \leq mt^{-r}=\frac{1}{r!}\left(  1-\frac{1}{t}\right)
\cdots\left(  1-\frac{r-1}{t}\right)  .\label{exf}%
\end{equation}

Let $\mathbf{x}:=\left(  x_{1},\ldots,x_{n}\right)  \in\Delta^{n-1}$ be an
eigenvector to $\mu\left(  G\right)  $ and suppose that $x_{1}\geq\cdots\geq
x_{n}$. If $x_{n}=0$, replace $G$ with the subgraph $H\subset G$ induced by
the vertices with nonzero entries in $\mathbf{x}$. Clearly $\mu\left(
H\right)  =\mu\left(  G\right)  $, and $H$ has at most $m$ edges. Since the
right side of (\ref{exf}) decreases with $t$, it is enough to prove
(\ref{exf}) for $H$. Thus, without loss of generality, we assume that
$\mathbf{x}$ is positive.

Set for short $\mu=\mu\left(  G\right)  $, and assume for a contradiction that
(\ref{exf}) fails, that is, $mt^{-r}<\mu$. Now, Lemma \ref{tsin} implies that
$mt^{-r}<\mu\leq m\sigma^{r}$, and so $\sigma>1/t$.

On the other hand, equation (\ref{eeq}) yields
\[
r\mu x_{1}=\frac{\partial P_{G}\left(  \mathbf{x}\right)  }{\partial x_{1}%
}x_{1}\leq P_{G}\left(  \mathbf{x}\right)  =\mu\text{;}%
\]
therefore, $x_{1}\leq1/r$. With this provision, and supposing that $3\leq
r\leq5$, Theorem \ref{tlub} gives
\begin{equation}
r\mu=\frac{\partial P_{G}\left(  \mathbf{x}\right)  }{\partial x_{1}}\leq
\frac{\partial S_{r}\left(  \mathbf{x}\right)  }{\partial x_{1}}\leq
r\sigma^{r}\binom{1/\sigma}{r}. \label{ine}%
\end{equation}
Hence, in view of $\sigma>1/t$, we find that
\[
\mu\leq\frac{1}{r!}\left(  1-\sigma\right)  \cdots\left(  1-\left(
r-1\right)  \sigma\right)  <\frac{1}{r!}\left(  1-\frac{1}{t}\right)
\cdots\left(  1-\frac{r-1}{t}\right)  =mt^{-r}.
\]
This contradiction shows that $\mu\left(  G\right)  \leq mt^{-r}$.

It remains to prove the conditions for equality in Conjecture \ref{con1}.
Suppose that $\mu\left(  G\right)  =mt^{-r}$; thus, equalities hold throughout
in (\ref{ine}), and by Theorem \ref{tlub} the entries of $\mathbf{x}$ are
equal to $1/n$. Therefore, we find that
\[
\frac{1}{r!}\left(  1-\frac{1}{t}\right)  \cdots\left(  1-\frac{r-1}%
{t}\right)  =mt^{-r}=\mu\left(  G\right)  \leq\binom{n}{r}n^{-r}=\frac{1}%
{r!}\left(  1-\frac{1}{n}\right)  \cdots\left(  1-\frac{r-1}{n}\right)  ,
\]
yielding in turn $n\geq t$. Now, $mt^{-r}=\mu\left(  G\right)  \leq mn^{-r}$
implies that $t=n$.

Finally, if $t$ is an integer, then taking $G$ to be the complete $r$-graph of
order $t$ and $\mathbf{x}\in\Delta^{t-1}$ to be the vector with all entries
equal to $1/t$, we see that $\mu_{r}\left(  m\right)  =mt^{-r}$, completing
the proof of Conjecture \ref{con1} for $3\leq r\leq5$.
\end{proof}

Our proof of Conjecture \ref{con1} for $t\geq4\left(  r-1\right)  \left(
r-2\right)  $ is similar, so we omit a few details.\medskip

\begin{proof}
[\textbf{Proof of Conjecture \ref{con1} for }$t\geq4\left(  r-1\right)
\left(  r-2\right)  .$]Let $G$ be an $r$-graph of order $n$ with $m$ edges,
and let $\mu\left(  G\right)  =\mu_{r}\left(  m\right)  .$ Suppose that
$t\geq4\left(  r-1\right)  \left(  r-2\right)  $ is a real number satisfying
$m=\binom{t}{r}$. To prove Conjecture \ref{con1}, we show that $\mu\left(
G\right)  \leq mt^{-r}$. To this end, select an eigenvector $\mathbf{x}%
:=\left(  x_{1},\ldots,x_{n}\right)  $ to $\mu\left(  G\right)  $ with
$x_{1}\geq\cdots\geq x_{n}>0$.

First note that if $x_{1}>\left(  r-1\right)  /t,$ then
\[
\frac{\left(  1-x_{1}\right)  ^{r-1}}{\left(  r-1\right)  !}<\frac{1}{\left(
r-1\right)  !}\left(  1-\frac{r-1}{t}\right)  ^{r-1}<\frac{1}{\left(
r-1\right)  !}\left(  1-\frac{1}{t}\right)  \cdots\left(  1-\frac{r-1}%
{t}\right)  =rmt^{-r},
\]
and therefore%
\[
r\mu\left(  G\right)  =\frac{\partial P_{G}\left(  \mathbf{x}\right)
}{\partial x_{1}}\leq\frac{\partial S_{r}\left(  \mathbf{x}\right)  }{\partial
x_{1}}<\frac{\left(  1-x_{1}\right)  ^{r-1}}{\left(  r-1\right)  !}<rmt^{-r},
\]
proving that $\mu\left(  G\right)  <mt^{-r}$.

On the other hand, if $x_{1}\leq\left(  r-1\right)  /t,$ then the premise
$t\geq4\left(  r-1\right)  \left(  r-2\right)  $ implies that $x_{1}%
\leq1/4\left(  r-2\right)  $; therefore, Theorem \ref{tub} yields
\[
\mu\left(  G\right)  =P_{G}\left(  \mathbf{x}\right)  \leq S_{r}\left(
\mathbf{x}\right)  \leq\sigma\binom{1/\sigma}{r}.
\]
This inequality and Lemma \ref{tsin} imply that $\mu\left(  G\right)  \leq
mt^{-r}$.

The proof of the condition for equality in $\mu\left(  G\right)  \leq mt^{-r}$
is omitted.
\end{proof}

\bigskip

\textbf{Acknowledgement. }A preliminary version of this work has been reported
at the International workshop on spectral hypergraph theory held in November,
2017 at Anhui University, Hefei, P.R. China. I am grateful to the organizers,
and particularly to Prof. Yi-Zheng Fan, for wonderful experience.

\end{document}